\documentclass[12pt]{amsart}
\usepackage{times,mathrsfs,color,comment}
\usepackage{amsmath}
\usepackage{amsfonts}
\usepackage{amssymb}
\usepackage{latexsym}
\usepackage{amsthm}
\usepackage{a4wide}

\usepackage{tikz}
\usetikzlibrary{calc,patterns}
\usetikzlibrary{shapes}
\usetikzlibrary{snakes}

\newtheorem{theorem}{Theorem}
\newtheorem{lemma}[theorem]{Lemma}

\newtheorem{coro}[theorem]{Corollary}

\newtheorem{proposition}[theorem]{Proposition}

\newcounter{other}            
\newtheorem{otherth}[other]{Theorem}              
\newtheorem{otherl}[other]{Lemma}        

\newcommand{\Cn}{\mathbb{C}^n}

\newcommand{\Sn}{\mathbb{S}_ n}

\newcommand{\Bn}{\mathbb{B}_ n}

\numberwithin{equation}{section}
\begin{document}

\title[Volterra operators from Bergman spaces to Hardy spaces]{Volterra type integration operators\\ from Bergman spaces to Hardy spaces}
\author[Santeri Miihkinen]{Santeri Miihkinen}
\address{Santeri Miihkinen \\Department of Mathematics \\
\AA bo Akademi University\\
FI-20500 \AA bo\\
Finland} \email{santeri.miihkinen@abo.fi}

\author[Jordi Pau]{Jordi Pau}
\address{Jordi Pau \\Departament de Matem\`{a}tiques i Inform\`{a}tica \\
Universitat de Barcelona\\
08007 Barcelona\\
Catalonia, Spain} \email{jordi.pau@ub.edu}

\author[Antti Per\"{a}l\"{a}]{Antti Per\"{a}l\"{a}}
\address{Antti Per\"{a}l\"{a} \\Departament de Matem\`{a}tiques i Inform\`{a}tica \\
Universitat de Barcelona\\
08007 Barcelona\\
Catalonia, Spain\\
Barcelona Graduate School of Mathematics (BGSMath).} \email{perala@ub.edu}

\author[Maofa Wang]{Maofa Wang}
\address{Maofa Wang \\School of Mathematics and Statistics, Wuhan University, Wuhan
430072,   CHINA \\} \email{mfwang.math@whu.edu.cn}

\date{March 29, 2019}


%


\thanks{J. Pau was partially
 supported by  DGICYT grant MTM2014-51834-P
(MCyT/MEC) and the grant 2017SGR358 (Generalitat de Catalunya). A. Per\"al\"a acknowledges financial support from the Spanish Ministry of Economy and Competitiveness, through the Mar\'ia de Maeztu Programme for Units of Excellence in R\&D (MDM-2014-0445). J. Pau and A. Per\"al\"a are also supported by the grant MTM2017-83499-P (Ministerio de Educaci\'{o}n y Ciencia). M. Wang was partially supported by  {NSFC (11771340)}  and China Scholarship Council.}


\begin{abstract}
We completely characterize the boundedness of the Volterra type integration operators $J_b$ acting from the weighted Bergman spaces $A^p_\alpha$ to the Hardy spaces $H^q$ of the unit ball of $\mathbb{C}^n$ for all $0<p,q<\infty$. A partial solution to the case $n=1$ was previously obtained by Z. Wu in \cite{Wu}. We solve the cases left open there and extend all the results to the setting of arbitrary complex dimension $n$. Our tools involve area methods from harmonic analysis, Carleson measures and Kahane-Khinchine type inequalities, factorization tricks for tent spaces of sequences, as well as techniques and integral estimates related to Hardy and Bergman spaces.
\end{abstract}

\maketitle


\section{Introduction}

{Let} $\Bn$ be the open unit ball of $\Cn$ and $H$ denote the algebra of holomorphic functions on $\Bn$. A function $b \in H$ induces an integration operator (or a Volterra operator) $J_b$ given by the formula:
\begin{equation}\label{Jb}
J_bf(z)=\int_0^1 f(tz)Rb(tz)\frac{dt}{t}, \quad z \in \Bn,
\end{equation}
where $f \in H$ and $Rb$ is the radial derivative of $b$:
$$Rb(z)=\sum_{k=1}^n z_k \frac{\partial b}{\partial z_k}(z),\quad z=(z_1, z_2, \dots, z_n) \in \Bn.$$
A fundamental property of the operator $J_b$ is the following basic formula involving the radial derivative $R$:
$$R(J_bf)(z)=f(z)Rb(z),\quad z \in \Bn.$$

\noindent A natural question about the integral operator is how to characterize $b$ so that $J_b$  is bounded from one holomorphic space to another. The operator $J_b$ was first studied by Pommerenke \cite{Pom} in the setting of the Hardy spaces of the unit disk and the functions of bounded mean oscillation. Some important papers include the pioneering works of Aleman, Cima and Siskakis \cite{AC, AS1, AS2}.  After then, a lot of reserach on $J_b$ has been done. The higher dimensional variant of $J_b$ was introduced by Hu \cite{Hu}. In the fairly recent paper \cite{P1} the second named author completely characterized the boundedness of $J_b$ between Hardy spaces $H^p$ and $H^q$ of the unit ball for the full range $0<p,q<\infty$. An interested reader may consult the paper for more references on the topic: this is an active field and there are too many relevant works to be listed here.

Recall that, for $0<p<\infty$, a function $f \in H$ belongs the Hardy space $H^p$, if
$$\|f\|_{H^p}^p=\sup_{0<r<1}\int_{\Sn}|f(r\xi)|^p d\sigma(\xi)<\infty.$$
Here $\Sn=\partial \Bn$ denotes the unit sphere, and $d\sigma$ the surface measure on $\Sn$ normalized so that $\sigma(\Sn)=1$.
Given $\alpha>-1$ and $0<p<\infty$, a function $f \in H$ belongs to the weighted Bergman space $A^p_\alpha$, if
$$\|f\|_{A^p_\alpha}^p = \int_{\Bn}|f(z)|^p dV_\alpha(z)<\infty.$$
Here $dV=dV_0$ is the Lebesgue measure on $\Bn$, normalized so that $V(\Bn)=1$. The measure $dV_\alpha$ is given by $dV_\alpha(z)=c(n,\alpha)(1-|z|^2)^\alpha dV(z)$ with $c(n,\alpha)$ chosen to guarantee $V_\alpha(\Bn)=1$. We also denote by $H^\infty$ the space of bounded holomorphic functions, equipped with the sup norm. As is well-known, all the spaces above are complete, and when $p=2$ they are Hilbert spaces with the obvious inner product. By now there are many excellent references for $H^p$ and $A^p_\alpha$ spaces; we refer the reader to \cite{ZZ, ZhuBn, Zhu}.

In the present paper, we characterize the boundedness of $J_b:A^p_\alpha \to H^q$ on the unit ball $\Bn$ for all possible ranges $0<p,q<\infty$ and $\alpha>-1$. For $n=1$ this problem was studied by Wu \cite{Wu} who solved the cases $0<p\leq q<\infty$ and $2\leq q<p<\infty$, and left the case $0<q<\min\{2,p\}$ as an open problem. We will fill the gap by solving the open case $0<q<\min\{2,p\}$ and extend all these results to the setting of higher dimension. \medskip

Our main theorem is the following.

\begin{theorem}\label{main}
Let $\alpha>-1$, $0<p,q<\infty$ and $b \in H$. Then the following hold:
\begin{itemize}
\item[(1) ] If $0<p\leq \min\{2,q\}$ or $2<p<q<\infty$, then $J_b:A^p_\alpha\to H^q$ is bounded if and only if
$$\sup_{z \in \Bn} |Rb(z)|(1-|z|^2)^{\frac{n}{q}+1-\frac{n+1+\alpha}{p}}<\infty.$$
\item[(2) ] If $2<p=q<\infty$, then $J_b:A^p_\alpha\to H^q$ is bounded if and only if $$|Rb(z)|^{\frac{2p}{p-2}}(1-|z|^2)^{\frac{p-2\alpha}{p-2}} dV(z)$$ is a Carleson measure.
\item[(3) ] If $p>\max\{2,q\}$, then $J_b:A^p_\alpha\to H^q$ is bounded if and only if
$$\xi \mapsto \left(\int_{\Gamma(\xi)}|Rb(z)|^{\frac{2p}{p-2}}
(1-|z|^2)^{\frac{2-2\alpha}{p-2}+1-n}dV(z)\right)^{\frac{p-2}{2p}}$$
belongs to $L^{\frac{pq}{p-q}}(\Sn)$.
\item[(4) ] If $0<q<p\leq 2$, then $J_b:A^p_\alpha\to H^q$ is bounded if and only if
$$\xi \mapsto \sup_{z \in \Gamma(\xi)}|Rb(z)| (1-|z|^2)^{\frac{p-1-\alpha}{p}}$$
belongs to $L^{\frac{pq}{p-q}}(\Sn)$.
\end{itemize}
\end{theorem}

The reader might find reading the theorem easier after examining Figure 1. The items of Theorem \ref{main} are labeled accordingly. The part of boundary shared by items (1) and (4) belongs to the item (1), and the part of boundary shared by items (3) and (4) belongs to the item (4).\\

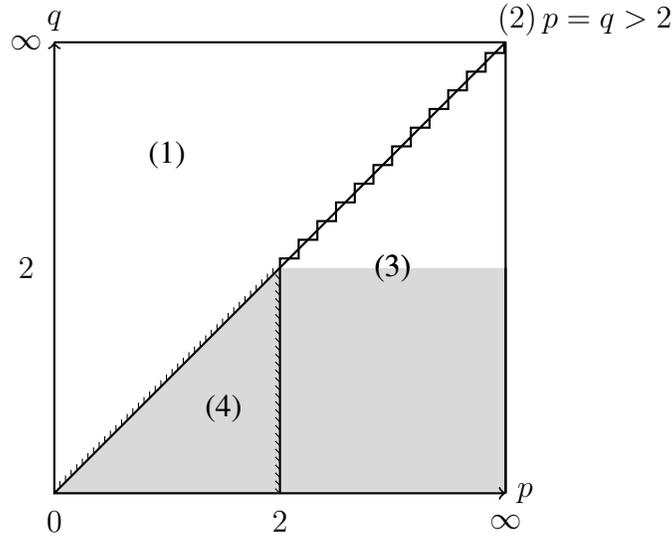
\begin{figure}[h]
\begin{center}
\caption{The main theorem summarized in $(p,q)$ coordinates}
\bigskip
\begin{tikzpicture}[scale=1.5, align = flush center, snake=zigzag]


    
    \draw (0,0) coordinate (a_1) -- (4,4) coordinate (a_2);
    
    \coordinate (c) at (2,2); 
     \coordinate (d) at (2,2);
    \coordinate (e) at (2,4);    
   
      \node at (-0.25,2) {$2$};
     \draw[thick] (2,0)--(d);
     \node at (2,-0.25) {$2$};
     \draw[thick, snake] (2,2) -- (4,4);
     \node at (-.25,4) {$\infty$};
     \node at (4,-.25) {$\infty$};
     \node at (1,3) {(1)};
     \node at (3,1) {(3)};
     \node at (3,2) {(3)};
     \node at (1.5,.75) {(4)};
     \node at (4.7,4.2) {$(2) \, p=q  >2$};
     \node at (0,-0.25) {$0$};
\draw (0,4) -- (0,0) -- (4,4) -- cycle;
\draw  (2,0) -- (2,2) -- (4,4) -- (4,0)-- cycle;    
   
\draw[fill=gray!30]  (0,0) -- (2,0) -- (2,2) -- cycle;   
\fill[fill=gray!30]  (2,2) -- (2,0) -- (4,0) -- (4,2)--cycle;

\draw[thick] (4,0)--(4,2);
\draw[thick] (0,0)--(4,4);
\draw[thick] (2,0)--(2,2);
\draw[thick] (0,4)--(4,4);
\draw[thick] (4,0)--(4,4);
\node at (1.5,.75) {(4)};
\draw [<->,thick] (0,4) node (yaxis) [above] {$q$}
        |- (4,0) node (xaxis) [right] {$p$};
\draw [snake=border,segment angle=45, segment amplitude = 2.5, segment length = 3] (0,0)--(2,2);
\draw [snake=border,segment angle=45, segment amplitude = 2.5, segment length = 3] (2,0)--(2,2);
\node at (3,2) {(3)};
\end{tikzpicture}
\end{center}
\end{figure}

We mention here the missing case $0<q<\min\{2,p\}$ is contained in the items (3) and (4).
In addition, aside from the cases when $J_b$ acts only on Hardy or Bergman spaces, the setting of the present paper is perhaps the most natural one, the complete solution is rather elaborate. All conditions above can be formulated in terms of some well-known function spaces, such as various tent spaces and Triebel-Lizorkin spaces. However, due to the multitude of the parameters, we believe that the actual conditions are easier to read. An exception is the item (1), which is equivalent to $f$ belonging to the Bloch type space $\mathcal{B}^{\frac{n}{q}+1-\frac{n+1+\alpha}{p}}$, whose definition is the necessary and sufficient condition given in item (1).

A reader who is familiar with the paper \cite{Wu} might notice that our case $2<p<q<\infty$ seems different from that of the reference. However, these conditions are known to be equivalent -- it is simply our personal choice to use a Bloch type semi-norm instead of $s$-Carleson measure condition. Our tools include Carleson measures, area techniques, Kahane-Khinchine type inequalities, and some factorization tricks for tent spaces of sequences.

The paper is organized as follows. The proof of the main theorem is split in two sections: Section \ref{pleq} containing the proof for $p\leq q$ (that is, cases (1) and (2)) and Section \ref{pbig} containing the case $p>q$ (that is, the items (3) and (4)). The sections preceeding these mains results contain some background materials and the tools used in the paper, as well as some key lemmas that are needed along the way.

We will use notation, which is quite standard. Typically constants are used with no attempt to calculate their exact values. Given two non-negative quantities $A$ and $B$, depending on some parameters, we write $A\lesssim B$ to imply that there exists some inessential constant $C>0$ so that $A\leq C B$. The converse relation $A\gtrsim B$ is defined in an analogous manner, and if $A\lesssim B$ and $A\gtrsim B$ both hold, we write $A \asymp B$. For two quasinormed spaces $X$ and $Y$, the notation $X\sim Y$ means that these spaces are isomorphic. Given $p \in [1,\infty]$, we will denote by $p'=p/(p-1)$ its H\"older conjugate. In this context we agree that $1'=\infty$ and $\infty'=1$.
\medskip

\section{Preliminaries}

In this section, we collect the necessary preliminaries for the course of the proof of our main theorem. The ideas listed in this section are used throughout the paper.

\subsection{Carleson measures and embedding theorems}

Let us recall the concept of a Carleson measure, which will be important for our analysis. For $\xi \in \Sn$ and $\delta>0$, consider the non-isotropic metric ball
$$B_{\delta}(\xi)=\left\lbrace z\in \Bn: |1-\langle z, \xi \rangle |<\delta \right\rbrace.$$
A positive Borel measure $\mu$ on $\Bn$ is said to be a Carleson measure if
$$\mu(B_{\delta}(\xi))\lesssim \delta^ n$$
for all $\xi \in \Sn$ and $\delta>0$. Obviously every Carleson measure is finite. H\"{o}rmander \cite{H} extended to several complex variables the famous Carleson measure embedding theorem \cite{Car0, Car1} asserting that, for $0<p<\infty$, the embedding $I_ d:H^p\rightarrow L^p(\mu):=L^p(\Bn,d\mu)$ is bounded if and only if $\mu$ is a Carleson measure.
More generally, for $s>0$, a finite positive Borel measure on $\Bn$ is called an $s$-Carleson measure if  $\mu(B_{\delta}(\xi))\lesssim \delta^{ns}$ for all $\xi\in \Sn$ and $\delta>0$. We set $$\|\mu\|_{CM_s}:=\sup_{\xi \in \Sn, \delta>0} \mu(B_{\delta}(\xi))\delta^{-ns}.$$
For simplicity, we write $\|\mu\|_{CM}$ for the $\|\mu\|_{CM_1}.$
It is well-known (see \cite[Theorem 45]{ZZ}) that $\mu$ is an $s$-Carleson measure if and only if for each (some) $t>0$,
\begin{equation}\label{sCM}
\sup_{a\in \Bn}\int_{\Bn} \!\!\frac{(1-|a|^2)^t}{|1-\langle a,z \rangle |^{ns+t}} \,d\mu(z)<\infty.
\end{equation}
Moreover, with constants depending on $t$, the supremum of the above integral is comparable to $\|\mu\|_{CM_s}$.
In \cite{Du}, Duren gave an extension of Carleson's theorem by showing that, for
 $0<p\leq q<\infty$, one has that $I_ d:H^p\rightarrow L^q(\mu)$ is bounded if and only if $\mu$ is a $q/p$-Carleson measure. Moreover, one has the estimate $$\|I_ d\|_{H^p\rightarrow L^q(\mu)}\asymp \|\mu\|^{1/q}_{CM_{q/p}}.$$ A simple proof of this result, in the setting of the unit ball, can be found in \cite{P1} for example.

We will need the following well-known embedding theorem for Hardy and Bergman spaces. The first part is a consequence of Duren's theorem, but it can be proven in a much more elementary manner (see Theorem 4.48 of \cite{ZhuBn}). The second part follows by writing $|f|^q=|f|^{q-p}|f|^p$ and using the standard pointwise estimate for the factor $|f|^{q-p}$:
$$|f(z)|^{q-p}\lesssim (1-|z|^2)^{-\frac{(n+1+\alpha)(q-p)}{p}}\|f\|_{A^p_\alpha}^{q-p}.$$

\begin{otherth}\label{embed}
Let $\alpha>-1$ and $0<p< q<\infty$. Then
$$H^p \subset A^q_{n(q/p-1)-1}$$
and
$$A^p_\alpha \subset A^q_{\beta},$$
with $\beta=(n+1+\alpha)(q/p-1)+\alpha$. Moreover, the inclusion mappings are bounded.
\end{otherth}

We will also need the following Dirichlet type embedding theorem, a result can be found in \cite{BB}.

\begin{otherth}\label{embed2}
Assume that $f \in H$ with $f(0)=0$. If $0<q\leq 2$, then
$$
\|f\|_{H^q}\lesssim \|Rf\|_{A^q_{q-1}}.
$$
If $0<p< q<\infty$, then

$$
\|f\|_{H^q}\lesssim \|Rf\|_{A^p_{p+np/q-n-1}}.
$$
\end{otherth}

\subsection{Area methods and equivalent norms}

For $\gamma>1$ and $\xi \in \Sn$, define the Kor\'anyi (admissible, non-tangential) approach region $\Gamma_\gamma(\xi)$ by
$$\Gamma_\gamma(\xi)=\left\lbrace z \in \Bn : |1-\langle z,\xi\rangle| <\frac{\gamma}{2}(1-|z|^2)\right\rbrace.$$
In this paper we agree that $\Gamma(\xi):=\Gamma_2(\xi)$. It is known that for every $r>1$ and $\gamma>1$, there exists $\gamma'>1$ so that
\begin{equation}\label{setin}
\bigcup_{z \in \Gamma_\gamma(\xi)}D(z,r)\subset \Gamma_{\gamma'}(\xi).
\end{equation}
We will write $\widetilde{\Gamma}(\xi)$ (and sometimes $\widetilde{\widetilde{\Gamma}}(\xi)$) to indicate this change of aperture.

Given $z \in \Bn$, we can define the set $I(z)=\{\xi \in \Sn : z \in \Gamma(\xi)\}\subset \Sn$. Since $\sigma(I(z))\asymp (1-|z|^2)^n$, an application of Fubini's theorem yields the following important formula:

\begin{equation}\label{EqG}
\int_{\Bn} \varphi(z)d\nu(z)\asymp \int_{\Sn} \left (\int_{\Gamma(\xi)} \varphi(z) \frac{d\nu(z)}{(1-|z|^2)^{n}} \right )d\sigma(\xi),
\end{equation}
where $\varphi$ is any positive measurable function and $\nu$ is a finite positive measure.

Let us now recall the following Hardy-Stein (or Littlewood-Paley) inequalities, which will be very important for our arguments. The proof of these variants can be found in \cite{ZhuBn}.

\begin{otherth}\label{HS}
Let $0<p<\infty$. If $f \in H$ and $f(0)=0$, then
$$\|f\|_{H^p}^p \asymp \int_{\Bn}|f(z)|^{p-2}|Rf(z)|^2 (1-|z|^2)dV(z).$$
In particular,
$$\|f\|_{H^2}^2 \asymp \int_{\Bn}|Rf(z)|^2 (1-|z|^2)dV(z).$$
\end{otherth}

The next estimate of the section is the celebrated Calder\'on's area theorem \cite{Cal}, which was proven for $1<p<\infty$ by Marcinkiewicz and Zygmund \cite{MZ}. The variant we will use can be found in \cite{AB} or in \cite{P1}, where it is proven in a more general form.

\begin{otherth}\label{area}
Let $0<p<\infty$. If $f \in H$ and $f(0)=0$, then
$$\|f\|_{H^p}^p \asymp \int_{\Sn}\left(\int_{\Gamma(\xi)}|Rf(z)|^2 (1-|z|^2)^{1-n}dV(z)\right)^{p/2}d\sigma(\xi).$$
\end{otherth}

The above result can be formulated in terms of $R^k f$ for any $k\in \mathbb{N}$, or even more general fractional derivatives. See, for instance \cite{PP, Per}. We will not need such formulas in this paper.

Given a positive Borel measure $\mu$ on $\Bn$, then for $\xi \in \Sn$, we define
$$\widetilde{\mu}(\xi)=\int_{\Gamma(\xi)}\frac{d\mu(z)}{(1-|z|^2)^n}.$$

\noindent The following result is known as Luecking's theorem, and is originally from \cite{Lue1}. The present variant can be found in \cite{P1}, for instance.

\begin{otherth}\label{LT}
Let $0<s<p<\infty$ and let $\mu$ be a positive Borel measure on $\Bn$. Then the identity $I_ d:H^p\rightarrow L^s(\mu)$ is bounded, if and only if, the function $\widetilde{\mu}$
belongs to $L^{p/(p-s)}(\Sn)$. Moreover,  $\|I_d\|_{H^p\rightarrow L^s(\mu)}\asymp \|\widetilde{\mu}\|_{L^{p/(p-s)}(\Sn)}^{1/s}.$
\end{otherth}

\subsection{Poisson transform and maximal functions}

Let us recall the invariant Poisson transform (see Chapter 4.1 of \cite{ZhuBn})
$$P[f](z)=\int_{\Sn}\frac{(1-|z|^2)^n}{|1-\langle z,\xi\rangle|^{2n}}f(\xi)d\sigma(\xi),$$
which is defined for $f \in L^1(\Sn)$.

Given $\delta>0$ and $\xi \in \Sn$, we define
$$I(\xi,\delta)=\{z \in \Sn: |1-\langle z,\xi\rangle|<\delta^2\},$$
and for $f \in L^1(\Sn)$ let $M$ denote the Hardy-Littlewood type maximal function
$$M[f](\xi)=\sup_{\delta>0}\frac{1}{\sigma(I(\xi,\delta))}\int_{I(\xi,\delta)}|f(\zeta)|d\sigma(\zeta).$$

The classical Hardy-Littlewood theorem (see Theorem 4.9 of \cite{ZhuBn} for this version) states that
\begin{equation}\label{hlm}
\|M[f]\|_{L^p(\Sn)}\lesssim \|f\|_{L^p(\Sn)},\quad p \in (1,\infty).
\end{equation}

We will need another kind of maximal function. For a continuous $f:\Bn\to \mathbb{C}$, we define the admissible (non-tangential) maximal function $N[f]$ by
$$N[f](\xi)=\sup_{z \in \Gamma(\xi)}|f(z)|.$$

Let us state the following well-known result on the $L^p$-boundedness of the admissible maximal function that can be found in \cite[Theorem 5.6.5]{Rud} or \cite[Theorem 4.24]{ZhuBn}.

\begin{otherth}\label{NTMT}
Let $0<p<\infty$ and $f\in H$. Then
$$ \|N[f]\|_{L^p(\Sn)}\lesssim \|f\|_{H^p}.$$
\end{otherth}

Finally, the following theorem (Theorem 4.10 of \cite{ZhuBn}) connects the Poisson transform and the two maximal functions.

\begin{otherth}\label{NPM}
For $f \in L^1(\Sn)$, we have
$$N[P[f]](\xi)\lesssim M[f](\xi).$$
\end{otherth}

\subsection{Kahane-Khinchine inequalities}

Consider a sequence of Rademacher functions $r_ k(t)$ (see \cite[Appendix A]{duren1}). For almost every $t\in (0,1)$ the sequence $\{r_ k(t)\}$ consists of signs $\pm 1$. We state first the classical Khinchine's inequality (see \cite[Appendix A]{duren1} for example).\medskip

\noindent \textbf{Khinchine's inequality:} Let $0<p<\infty$. Then for any sequence $\{c_k\}$ of complex numbers, we have
\begin{equation}\label{khin}
\left(\sum_k |c_k|^2\right)^{p/2}\asymp \int_0^1 \left|\sum_k c_kr_k(t)\right|^p dt.
\end{equation}

The next result is known as Kahane's inequality, and it will be usually applied in connection to Khinchine's inequality. For reference, see for instance Lemma 5 of Luecking \cite{Lue2} or the paper of Kalton \cite{Kal}.\medskip

\noindent \textbf{Kahane's inequality:} Let $X$ be a Banach space,  and $0<p,q<\infty$. For any sequence $\{x_ k\}\subset X$, one has
\begin{equation}\label{kah}
\left ( \int_{0}^1 \Big \|\sum_ k r_ k(t)\, x_ k \Big \|_{X}^q dt\right )^{1/q} \asymp \left ( \int_{0}^1 \Big \|\sum_ k r_ k(t)\, x_ k \Big \|_{X}^p dt\right )^{1/p}.
\end{equation}
Moreover, the implicit constants can be chosen to depend only on $p$ and $q$, and not on the Banach space $X$.

\subsection{Separated sequences and lattices}

Let $\beta(z,w)$ denote the Bergman metric on $\Bn$, and $D(a,r)=\{z\in \Bn : \beta(a,z)<r\}$ be the Bergman metric ball of radius $r>0$ centered at a point $a\in \Bn$.
A sequence of points $Z=\{a_k\}\subset \Bn$ is said to be separated if there exists $\delta>0$ such that $\beta(a_ i,a_ j)\geq \delta$ for all $i$ and $j$ with $i\neq j$. This implies that there is $r>0$ such that the Bergman metric balls $D(a_k,r)=\{z\in \Bn :\beta(z,a_k)<r\}$ are pairwise disjoint.

We need a well-known
result on decomposition of the unit ball $\Bn$.
By Theorem 2.23 in \cite{ZhuBn},
there exists a positive integer $N$ such that for any $0<r<1$ we can
find a sequence $\{a_k\}$ in $\Bn$ with the following properties:
\begin{itemize}
\item[(i)  ] $\Bn=\bigcup_{k}D(a_k,r)$;
\item[(ii) ] The sets $D(a_k,r/4)$ are mutually disjoint;
\item[(iii) ] Each point $z\in\Bn$ belongs to at most $N$ of the sets $D(a_k,4r)$.
\end{itemize}

\noindent Any sequence $\{a_k\}$ satisfying the above conditions is called
an $r$-\emph{lattice}
(in the Bergman metric). Obviously any $r$-lattice is a separated sequence.

Our main application of separated sequences is in context of the following important result.  It is essentially due to Coifman and Rochberg \cite{CR}, and can be found in Theorem 2.30 of \cite{ZhuBn}. Note that we only need one part of the cited theorem, and it is easily seen to be valid for all separated sequences.

\begin{otherth}
Let $0<p<\infty$, $\alpha>-1$, and
$s>n\max\{1,1/p\}-n/p.$
For any separated sequence $\{a_k\}$, and $\lambda=\{\lambda_k\} \in \ell^p$, the function
$$f(z)=\sum_{k}\lambda_k \frac{(1-|a_k|^2)^s}{(1-\langle z,a_k\rangle)^{s+\frac{n+1+\alpha}{p}}}$$
belongs to $A^p_\alpha$, where the series converges in the quasinorm topology of $A^p_\alpha$. Moreover,
$\|f\|_{A^p_\alpha}\lesssim \|\lambda\|_{\ell^p}.$
\end{otherth}

\subsection{Tent spaces}

Tent spaces were introduced in the paper of Coifman, Meyer and Stein \cite{CMS} to study problems in harmonic analysis. They provide us with simple, yet general framework for questions regarding important spaces such Hardy spaces, Bergman spaces and BMOA. Luecking \cite{Lue1} used tent spaces to study embedding theorems for Hardy spaces on $\mathbb{R}^n$. These results have been translated to $\Bn$ by Arsenovic and Jevtic \cite{Ars, Jev}.

Let $0<p,q<\infty$ and $\nu$ be a positive Borel measure. The tent space $T^{p}_{q,\nu}$ consists of $\nu$-measurable functions $f$ with
$$\|f\|_{T^{p}_{q,\nu}}^p : =\int_{\Sn} \left(\int_{\Gamma(\xi)} |f(z)|^q d\nu(z)\right)^{\frac{p}{q}}d\sigma(\xi)<\infty.$$
Analogously, the space $T^{p}_{\infty,\nu}$ consists of $\nu$-measurable $f$ with
$$\|f\|_{T^{p}_{\infty,\nu}}^p : =\int_{\Sn}\left(\operatorname{ess}\sup_{z \in \Gamma(\xi)}|f(z)|\right)^p d\sigma(\xi)<\infty.$$
The essential supremum is taken with respect to the measure $\nu$. The case $p=\infty$ is different. For non-zero $u\in\Bn$, we define $\zeta_u=u/|u|$ and set $Q(u)=\{z\in \Bn: |1-\langle z,\zeta_u\rangle|<1-|u|^2\}$. We agree that $Q(0)=\Bn$. The space $T_{q,\nu}^\infty$ consists of $\nu$-measurable functions $f$ with
$$\|f\|_{T_{q,\nu}^\infty}=\operatorname{ess} \sup_{\xi\in\Sn}\left(
\sup_{u\in\Gamma(\xi)}{1\over (1-|u|^2)^n}\int_{Q(u)}|f(z)|^q(1-|z|^2)^{n}d\nu(z)\right)^{1/q}<\infty.$$
By comparing with the discussion in Section 5.2 of \cite{ZhuBn}, we notice that $f\in T_{q,\nu}^\infty$ if and only if $d\mu_{f}(z)=(1-|z|^2)^{n}|f(z)|^q d\nu(z)$ is a Carleson measure on $\Bn$. Moreover, $\|f\|_{T^{\infty}_{q,\nu}} \asymp \|\mu_{f}\|_{CM_ 1}^{1/q}$.

The aperture $\gamma>0$ of the Kor\'anyi region is suppressed from the notation, and it is well-known that any two apertures generate the same function space with equivalent quasinorms. When $\nu(z)=(1-|z|^2)^\alpha$, we write $T^p_{q,\nu}=T^p_{q,\alpha}$. {The spaces $T^p_{\infty,\alpha}$ are obviously independent of $\alpha$}, and we simpy write $T^p_\infty$. Note that Theorem \ref{area} states that $f \in H$ belongs to $H^p$ if and only if $Rf \in T^p_{2,1-n}$, and Fubini's theorem shows that $A^p_\alpha$ consists of holomorphic functions in $T^p_{p,\alpha-n}$. This explains the special role of the number  $2$ in Theorem \ref{main}.

We will employ a discretization scheme. For this purpose a particularly important case is given when $\nu=\sum_k \delta_{a_k}$, where $Z=\{a_k\}$ is a separated sequence and $\delta_{a_k}$ are the usual Dirac point masses at points $a_k$. We denote $f(a_k)=\lambda_k$, and say that $\lambda=\{\lambda_k\} \in T^{p}_q(Z)$, if
$$\|\lambda\|_{T^{p}_q(Z)}^p : =\int_{\Sn} \left(\sum_{a_k \in \Gamma(\xi)}|\lambda_k|^q  \right)^{\frac{p}{q}}d\sigma(\xi)<\infty,$$
when $0<p,q<\infty$, and $\lambda=\{\lambda_k\} \in T^{p}_\infty (Z)$, if
$$\|\lambda\|_{T^{p}_\infty (Z)}^p : =\int_{\Sn}\left(\sup_{a_k \in \Gamma(\xi)}|\lambda_k|\right)^p d\sigma(\xi)<\infty.$$

\noindent Finally, $\lambda=\{\lambda_k\} \in T^{\infty}_q(Z)$, if
$$\|\lambda\|_{T_{q}^\infty(Z)}=\operatorname{ess} \sup_{\xi\in\Sn}\left(
\sup_{u\in\Gamma(\xi)}{1\over (1-|u|^2)^n}\sum_{a_k \in Q(u)}|\lambda_k|^q (1-|a_k|^2)^{n}\right)^{1/q}<\infty.$$
As before, we have that $\lambda \in T^{\infty}_q(Z)$ if and only if the measure $d\mu_{\lambda}=\sum_ k |\lambda_ k|^q (1-|a_ k|^2)^n \delta_{a_ k}$ is a Carleson measure. Moreover,  $\|\lambda\|_{T_{q}^\infty(Z)}\asymp \|\mu_{\lambda}\|_{CM_ 1}^{1/q}$. \\

\noindent We will need the following duality result for the tent spaces of sequences. For the proof, see \cite{Ars, Jev, Lue1}.

\begin{otherth}\label{TTD1}
Let $1<p<\infty$ and $Z=\{a_k\}$ be a separated sequence. If $1<q<\infty$, then the dual of $T^p_q(Z)$ is isomorphic to $T^{p'}_{q'}(Z)$ under the pairing:
$$\langle \lambda ,\mu \rangle_{T^2_2(Z)} =\sum_ k \lambda _ k \,\overline{\mu_ k} (1-|a_ k|^2)^n,\quad \lambda \in T^p_q(Z),\quad \mu \in T^{p'}_{q'}(Z).$$
If $0<q\leq 1$, then the dual of $T^p_q(Z)$ is isomorphic to $T^p_\infty(Z)$ under the same pairing.
\end{otherth}

\subsection{Estimates involving Bergman kernels}

Let us next recall the following well-known (Forelli-Rudin) integral estimates that have become very useful in this area of analysis (see \cite[Theorem 1.12]{ZhuBn} for example).
\begin{otherl}\label{IctBn}
Let $t>-1$ and $s>0$. Then
$$\int_{\Sn} \frac{d\sigma(\xi)}{|1-\langle z,\xi\rangle |^{n+s}} \lesssim (1-|z|^2)^{-s}$$
and
$$ \int_{\Bn} \frac{(1-|u|^2)^t\,dV(u)}{|1-\langle z,u\rangle |^{n+1+t+s}}\lesssim (1-|z|^2)^{-s}$$
for all $z\in \Bn$.
\end{otherl}

We also need the following discrete version of the previous estimate.

\begin{otherl}\label{l2}
Let $\{a_k\}$ be a separated sequence in $\Bn$ and let $n<t<s$.
Then
$$
\sum_{k}\frac{(1-|a_k|^2)^t}{|1-\langle z,a_k \rangle |^s}\lesssim
(1-|z|^2)^{t-s},\qquad z\in \Bn.
$$
\end{otherl}

The following more general version of Lemma \ref{IctBn} will be used. The proof can be found in \cite{OF}.

\begin{otherl}\label{FRgeneral}
Let $s>-1$, $s+n+1>r,t>0$, and $r+t-s>n+1$. For $a \in \Bn$ and $z\in \overline{\mathbb{B}}_ n$, one has
$$\int_{\Bn} \frac{(1-|u|^2)^s dV(u)}{|1-\langle z,u\rangle|^r |1-\langle a,u\rangle|^t} \lesssim \frac{1}{|1-\langle z,a\rangle|^{r+t-s-n-1}}.$$
\end{otherl}

Finally, we will use an integral estimate connecting the Kor\'anyi regions and the Bergman type kernels. Its original form is due to Luecking \cite{Lue1}, and the present version can be found in \cite{Ars} and \cite{Jev}. Note that the converse inequality also holds, and is quite trivial to prove.

\begin{otherl}\label{Gamma}
Let $0<s<\infty$ and $\theta>n\max\{1,1/s\}$. If $\mu$ is a positive measure, then
$$\int_{\Sn}\left[\int_{\Bn}\left(\frac{1-|z|^2}{|1-\langle z,\xi\rangle|}\right)^\theta d\mu(z)\right]^s d\sigma(\xi)\lesssim \int_{\Sn}\mu(\Gamma(\xi))^s d\sigma(\xi).$$
\end{otherl}

\section{Lemmas on discretization}

We will on several occasions use Kahane's and Khinchine's inequalities throughout the proof of our main theorem. These tools provide discretized versions of the conditions we really need. In this section, we show how to obtain the continuous characterizations from the discrete ones. For Bergman space analogues of these lemmas, see \cite{Lu85}.

\begin{lemma}\label{disc_cont}
Let $0<p,q<\infty$ and $\beta>-n-1$. There exists $r_0 \in (0,1)$ so that, if $0<r<r_0$ and $Z=\{a_k\}$ is an $r$-lattice, then
\begin{equation*}
\int_{\Sn}\left(\int_{\Gamma(\xi)}|f(z)|^q (1-|z|^2)^\beta dV(z)\right)^{\frac{p}{q}}d\sigma(\xi) \lesssim \int_{\Sn}\left(\sum_{a_k \in \Gamma(\xi)} |f(a_k)|^q (1-|a_k|^2)^{n+1+\beta}\right)^{\frac{p}{q}}d\sigma(\xi),
\end{equation*}
whenever $f$ is holomorphic and in $T^p_{q,\beta}$.
\end{lemma}

\begin{proof}
It is clear that
$$\int_{\Gamma(\xi)}|f(z)|^q (1-|z|^2)^\beta dV(z) \lesssim \sum_{a_k \in \widetilde{\Gamma}(\xi)}\int_{D(a_k,r)}|f(z)|^q (1-|z|^2)^\beta dV(z),$$
where $\widetilde{\Gamma}(\xi)$ is an approach region with a larger aperture.
By using the standard reproducing formula for $f(z)-f(0)$ followed by a change of variables with the corresponding ball automorphism (compare with the proof of Theorem 2.3 in \cite{Lu85}  {or see Lemma 2.2 in \cite{KW})}, we obtain for {$z,u$ with $\beta(z,u)<r<1/2$} that
$$|f(z)-f(u)|^q \lesssim r^q \int_{D(u,1)} |f(w)|^p \frac{(1-|u|^2)^{n+1}}{|1-\langle w,u\rangle|^{2n+2}}dV(w),$$
where the constant implied by "$\lesssim$" does not depend on $r$.
Thus, an application of Fubini's theorem gives us
\begin{align*}
\sum_{a_k \in \widetilde{\Gamma}(\xi)}\int_{D(a_k,r)}&|f(z)-f(a_k)|^q (1-|z|^2)^\beta dV(z)\\
\lesssim &\sum_{a_k \in \widetilde{\Gamma}(\xi)} \int_{D(a_k,r)} r^q \int_{D(a_k,1)} |f(w)|^q \frac{(1-|a_k|^2)^{n+1}}{|1-\langle w,a_k\rangle|^{2n+2}}dV(w)(1-|z|^2)^\beta dV(z)\\
\lesssim & r^q\sum_{a_k \in \widetilde{\Gamma}(\xi)} \int_{D(a_k,1)} |f(w)|^q (1-|w|^2)^\beta dV(w)\\
\lesssim & r^q \int_{\widetilde{\widetilde{\Gamma}}(\xi)} |f(w)|^q (1-|w|^2)^\beta dV(w).
\end{align*}
The last inequality is due to (\ref{setin}), and
$\widetilde{\widetilde{\Gamma}}(\xi)$ is an approach region with aperture larger than $\widetilde{\Gamma}(\xi)$. {Therefore,}
\begin{align*}
\sum_{a_k \in \widetilde{\Gamma}(\xi)}&\int_{D(a_k,r)}|f(z)|^q (1-|z|^2)^\beta dV(z)\\
\lesssim & \sum_{a_k \in \widetilde{\Gamma}(\xi)}\int_{D(a_k,r)}|f(z)-f(a_k)|^q (1-|z|^2)^\beta dV(z)+\sum_{a_k \in \widetilde{\Gamma}(\xi)}|f(a_k)|^q (1-|a_k|^2)^{n+1+\beta}\\
\lesssim & r^q \int_{\widetilde{\widetilde{\Gamma}}(\xi)} |f(w)|^q (1-|w|^2)^\beta dV(w)+\sum_{a_k \in \widetilde{\Gamma}(\xi)}|f(a_k)|^q (1-|a_k|^2)^{n+1+\beta}.
\end{align*}
Recall that different apertures define the same tent spaces with equivalent quasinorms. Raising to power $p/q$ and integrating over $\Sn$ yields
\begin{align*}
\int_{\Sn}&\left(\int_{\Gamma(\xi)}|f(z)|^q (1-|z|^2)^\beta dV(z)\right)^{p/q} \!\!\!d\sigma(\xi)\\
&\lesssim \int_{\Sn}\left(\sum_{a_k \in \Gamma(\xi)} |f(a_k)|^q (1-|a_k|^2)^{n+1+\beta}\right)^{p/q} \!\!\!d\sigma(\xi)\\
&+r^{p}\int_{\Sn}\left(\int_{\Gamma(\xi)}|f(z)|^q (1-|z|^2)^\beta dV(z)\right)^{p/q}d\sigma(\xi).
\end{align*}
Since the constants in "$\lesssim$" do not depend on $r$, we will find the desired $r_0$, which completes the proof.
\end{proof}

By a similar argument, we obtain the following lemma.

\begin{lemma}\label{discsup}
Let $0<p<\infty$ and $\beta\geq 0$. There exists $r_0 \in (0,1)$ so that, if $0<r<r_0$ and $Z=\{a_k\}$ is an $r$-lattice, then
\begin{equation*}
\int_{\Sn}\sup_{z\in \Gamma(\xi)}|f(z)|^p (1-|z|^2)^\beta d\sigma(\xi) \lesssim \int_{\Sn}\sup_{a_k \in \Gamma(\xi)} |f(a_k)|^p (1-|a_k|^2)^{\beta}d\sigma(\xi),
\end{equation*}
whenever $f$ is holomorphic and the left-hand-side is finite.
\end{lemma}

\begin{proof}
Note first that
\begin{align*}
|f(z)|^p(1-|z|^2)^\beta &\lesssim \frac{1}{(1-|z|^2)^{n+1}}\int_{D(z,r)}|f(w)|^p (1-|w|^2)^\beta dV(w) \\
&\lesssim \frac{1}{(1-|z|^2)^{n+1}}\int_{D(a_k,2r)}|f(w)|^p (1-|w|^2)^\beta dV(w),
\end{align*}
for some $a_k$ in the lattice. Therefore, for $0<r<1/4$, we have
$$\sup_{z \in \Gamma(\xi)} |f(z)|^p(1-|z|^2)^\beta \lesssim \sup_{a_k \in \widetilde{\Gamma}(\xi)} \frac{1}{(1-|a_k|^2)^{n+1}}\int_{D(a_k,2r)}|f(w)|^p (1-|w|^2)^\beta dV(w).$$
A small modification of the corresponding argument in the previous lemma gives us
\begin{displaymath}
\begin{split}
\sup_{z \in \Gamma(\xi)} |f(z)|^p &(1-|z|^2)^\beta \\
&\lesssim r^p \sup_{z \in \widetilde{\widetilde{\Gamma}}(\xi)} |f(z)|^p(1-|z|^2)^\beta+\sup_{a_k \in \widetilde{\Gamma}(\xi)} |f(a_k)|^p (1-|a_k|^2)^{\beta}.
\end{split}
\end{displaymath}
Integrating over $\Sn$ gives the result, by the same reasoning as before.
\end{proof}

\smallskip

\section{The case $p\leq q$}\label{pleq}

In this section we prove the items (1) and (2) of our main theorem. When $n=1$, these result can be found in the paper of Wu \cite{Wu}. However, our proofs differ substantially from his. We also note that Theorem \ref{pleqmin} is formulated in terms of a Bloch type condition, which we find easier to use in practice than the $s$-Carleson condition in \cite{Wu}. Let us start with the following.

\begin{theorem}\label{pleqmin}
Let $0<p\leq \min\{2,q\}$ or  $2<p<q<\infty$ and $\alpha>-1$. Then $J_b:A^p_\alpha\to H^q$ is bounded if and only if
$$\|b\|_{\mathcal{B}^{\gamma}}:=\sup_{z\in \Bn} |Rb(z)|(1-|z|^2)^{\gamma}<\infty.$$
with $\gamma=\frac{n}{q}+1-\frac{(n+1+\alpha)}{p}$. Moreover,  $\|J_b\|_{A_\alpha^p\to H^q}\asymp \|b\|_{\mathcal{B}^{\gamma}}.$
\end{theorem}

\begin{proof}
Let us first consider the necessity. By the standard pointwise estimate for the derivative of Hardy space functions, we have
$$|f(z)||Rb(z)|\lesssim (1-|z|^2)^{-n/q-1}\|J_b\|_{A^p_\alpha\to H^q}\cdot \|f\|_{A^p_\alpha}.$$
Now, taking
$$f(u)=f_ z(u)=\frac{(1-|z|^2)^{s-\frac{n+1+\alpha}{p}}}{(1-\langle u,z\rangle)^s}$$
for large enough $s$ gives that
\begin{equation*}\label{bloch}
\sup_{z \in \Bn}|Rb(z)|(1-|z|^2)^{\frac{n}{q}+1-\frac{n+1+\alpha}{p}}\lesssim\|J_b\|_{A_\alpha^p\to H^q}<\infty.
\end{equation*}
and the necessity is proven.

Next, we look for the sufficiency. Note that $p$ and $q$ of the theorem satisfy either $p<q$ or $p= q\leq 2$. In the former case,  by using the second part of Theorem \ref{embed2},
\begin{align*}
\|J_b f\|_{H^q}&\lesssim \|R(J_b f)\|_{A^p_{p+np/q-n-1}}\\
&=\left(\int_{\Bn}|f(z)|^p |Rb(z)|^p (1-|z|^2)^{p+np/q-n-1}dV(z)\right)^{1/p}\\
&\lesssim \|b\|_{\mathcal{B}^{\gamma}} \cdot \|f\|_{A^p_\alpha}.
\end{align*}
In the case $q=p\leq 2$, we  use the first part of Theorem  \ref{embed2}:
\begin{align*}
\|J_b f\|_{H^p}&\lesssim \|R(J_b f)\|_{A^p_{p-1}}=\left(\int_{\Bn}|f(z)|^p |Rb(z)|^p (1-|z|^2)^{p-1}dV(z)\right)^{1/p}\\
&\lesssim \|b\|_{\mathcal{B}^{1-\frac{(1+\alpha)}{p}}} \cdot \|f\|_{A^p_\alpha}.
 \end{align*}
 Therefore, in both cases,
\begin{equation*}\label{inversebloch}
\|J_b\|_{A_\alpha^p\to H^q}\lesssim\|b\|_{\mathcal{B}^{\gamma}}.
\end{equation*}

\noindent This completes the proof.
\end{proof}

We next prove the item (2) of our main theorem.

\begin{theorem}\label{2<p=q}
Let $2<p<\infty$ and $\alpha>-1$. Then $J_b:A^p_\alpha\to H^p$ is bounded if and only if
$$d\mu_b(z):=|Rb(z)|^{\frac{2p}{p-2}}(1-|z|^2)^{\frac{p-2\alpha}{p-2}} dV(z)$$ is a Carleson measure. Moreover, $\|J_b\|_{A_\alpha^p\to H^p}\asymp\|\mu_b\|_{CM_ 1}^{\frac{p-2}{2p}}.$
\end{theorem}

\begin{proof}
For any $f\in A_\alpha^p,$ consider the measure $d\mu_{f,b}(z)=|f(z)|^2|Rb(z)|^2
dV_1(z)$. Then
  \begin{align}\label{equi}
 \|J_bf\|_{H^p}^p\asymp\int_{\Sn}\left(\int_{\Gamma(\xi)}|f(z)|^2|Rb(z)|^2
 {dV_1(z)\over(1-|z|^2)^n}\right)^{p/2}d\sigma(\xi)
 =\int_{\Sn}\widetilde{\mu_{f,b}}^{p/2}(\xi)d\sigma(\xi).
 \end{align}

We now first prove the sufficiency part.  \noindent For any $h \in H^{\frac{p}{p-2}}$, we have by H\"older's inequality
\begin{align*}
\int_{\Bn}|h(z)|d\mu_{f,b}(z)=&\int_{\Bn}|h(z)||f(z)|^2|Rb(z)|^2(1-|z|^2)dV(z)\\
 \lesssim&
 \left(\int_{\Bn}|h(z)|^{p\over p-2}|Rb(z)|^{2p\over p-2}(1-|z|^2)^{p-2\alpha\over p-2}dV(z)\right)^{p-2\over p}\|f\|_{A_\alpha^p}^2\\
 \lesssim& \|h\|_{H^{\frac{p}{p-2}}} \cdot \|\mu_b\|_{CM_ 1}^{p-2\over p} \cdot \|f\|_{A^p_\alpha}^2.
 \end{align*}
 The last inequality is due to $d\mu_b$ being a Carleson measure. Therefore, we proved that the identity $I_d: H^{p\over p-2}\to L^1(d\mu_{f,b})$ is bounded.
 By Luecking's theorem (i.e., Theorem \ref{LT}), we have
 \[
 \int_{\Sn}\widetilde{\mu_{f,b}}^{p/2}(\xi)d\sigma(\xi)<\infty
 \]
  with
 \[
  \|\widetilde{\mu_{f,b}}\|_{L^{p\over 2}(d\sigma)}\asymp\|I_d\|_{H^{p\over p-2}\to L^1(d\mu_{f,b})}\lesssim \|\mu_b\|_{CM_ 1}^{p-2\over p}\cdot \|f\|_{A^p_\alpha}^2.
 \]
 Thus $J_b:A_\alpha^p\to H^p$ is bounded and $\|J_b\|_{A_\alpha^p\to H^p}\lesssim \big \|\mu_b \big \|_{CM_ 1}^{\frac{p-2}{2p}}.$
 \noindent This concludes the proof of the sufficiency.

 Let us next consider the necessity. Assume that $J_b:A^p_\alpha\to H^p$ is bounded. Then for any $f\in A^p_\alpha$, by (\ref{equi}), we have
 \begin{equation*}
 \int_{\Sn}\widetilde{\mu_{f,b}}(\xi)^{p/2}d\sigma(\xi)
 \lesssim\|J_b\|_{A_\alpha^p\to H^p}^p \cdot \|f\|_{A_\alpha^p}^p<\infty.
 \end{equation*}
 We see that the measure $d\mu_{f,b}$ satisfies the conditions of Theorem \ref{LT} for parameters $1$ and ${p\over p-2}$. It follows that for any $h\in H^{p\over p-2}$, we have
 \begin{align*}
 \int_{\Bn}|h(z)|d\mu_{f,b}(z) \lesssim\|J_b\|_{A_\alpha^p\to H^p}^2\cdot \|f\|_{A_\alpha^p}^2 \cdot \|h\|_{H^{p\over p-2}}.
 \end{align*}
 Now define the measure $d\nu_{h,b}(z)=|h(z)||Rb(z)|^2dV_1(z)$, and then the estimate above can be written as:

\begin{align*}
 \int_{\Bn} |f(z)|^2 d\nu_{h,b}(z)\lesssim\|J_b\|_{A_\alpha^p\to H^p}^2 \cdot \|f\|_{A_\alpha^p}^2 \cdot \|h\|_{H^{p\over p-2}}.
 \end{align*}
 This is a Carleson measure condition for the Bergman spaces. Therefore, since $p>2$, by the results in  \cite{Lue2} (see also Theorem B in \cite{PZ}), we have for any $r>0$
 $$\hat{\nu}_{h,b}(z):={\nu_{h,b}(D(z,r))\over (1-|z|^2)^{n+\alpha+1}}\in L^{p\over p-2}(\Bn,dV_\alpha)$$
 and $\|\hat{\nu}_{h,b}\|_{L^{p\over p-2}(dV_\alpha)}\lesssim\|J_b\|^2_{A_\alpha^p\to H^p}\,\|h\|_{H^{p\over p-2}}.$
By subharmonicity, we have $$\nu_{h,b}(D(z,r))\gtrsim|h(z)||Rb(z)|^2(1-|z|^2)^{n+2},$$ and this estimate  together with corresponding norm inequalities give us 
\begin{equation}\label{inserth}
\int_{\Bn}|h(z)|^{p\over p-2}|Rb(z)|^{2p\over p-2}(1-|z|^2)^{p-2\alpha\over p-2}dV(z)\lesssim\|J_b\|_{A_\alpha^p\to H^p}^{\frac{2p}{p-2}}\cdot\|h\|_{H^{p\over p-2}}^{\frac{p}{p-2}}.
\end{equation}
For $a \in \Bn$, we define $h:=h_a(z)=(1-\langle z,a\rangle)^{-\beta}$, where $\beta>n(p-2)/p$. By the standard integral estimates
$$\|h_a\|_{H^{p\over p-2}}^{\frac{p}{p-2}} \asymp \frac{1}{(1-|a|^2)^{\frac{\beta p}{p-2}-n}}.$$

\noindent Inserting this into \eqref{inserth} shows that $d\mu_b$ satisfies the condition \eqref{sCM} with $s=1$. Therefore $d\mu_b$  is a Carleson measure and $ \|\mu_b\|_{CM_ 1}^{\frac{p-2}{2p}}\lesssim\|J_b\|_{A_\alpha^p\to H^p}$, which completes the proof of the necessity.

\end{proof}

\section{The case $p>q$}\label{pbig}

We have proven the items (1) and (2) of Theorem \ref{main} in the previous section. In this section, we will use a factorization trick for the tent spaces of sequences to solve the remaining case for $p>q$. Here we also solve the case $0<q<\min\{2,p\}$, which is open  in \cite{Wu}. It is contained in Theorems \ref{th1} and \ref{th11}.

We recall a fact about the factorization of quasinormed spaces of sequences $X,Y$ and $W$. We say that $W=X\cdot Y$, if for $x=\{x_k\} \in X$ and $y=\{y_k\} \in Y$, we have $\|\{x_k y_k\}\|_W \lesssim \|x\|_X \|y\|_Y$, and every $w=\{w_k\} \in W$ can be expressed as $\{w_k\}=\{x_k y_k\}$ with $\|w\|_W \gtrsim \inf \|x\|_X \|y\|_Y$, where the infimum is taken over all possible factorizations of $w$.  Note that then $\tau \in W^*$, if
$$|\tau(\{x_k y_k\})|\lesssim \|x\|_X \|y\|_Y$$
for every $x \in X$ and $y \in Y$.

We will use the following result concerning factorization of sequence tent spaces. A similar result for tent spaces of functions over the upper half-space was proven in \cite{CV} by Cohn and Verbitsky. This version is also likely known to specialists, but we were unable to find a proof in the literature.

\begin{proposition}\label{factor}
Let $0<p,q<\infty$ and $Z=\{a_k\}$ be an $r$-lattice. If $p<p_1, p_2<\infty$, $q<q_1, q_2<\infty$ and satisfy
$$\frac{1}{p_1}+\frac{1}{p_2}=\frac{1}{p}, \quad \text{and} \quad \frac{1}{q_1}+\frac{1}{q_2}=\frac{1}{q}.$$
Then $$T^p_q(Z)=T^{p_1}_{q_1}(Z)\cdot T^{p_2}_{q_2}(Z).$$
\end{proposition}

\begin{proof}
It suffices to prove
\begin{equation}\label{infinf}
T^p_q(Z)=T^p_\infty(Z)\cdot T^\infty_q(Z), \quad 0<p,q<\infty.
\end{equation}
Based on this we can obtain the desired result  by using the following facts
$$T^p_\infty(Z)=T^{p_1}_\infty(Z)\cdot T^{p_2}_\infty(Z), \qquad T^\infty_q(Z)=T^\infty_{q_1}(Z)\cdot T^\infty_{q_2}(Z).$$

So let us prove \eqref{infinf}. We follow the idea from \cite{CV}. We begin by showing that, for sequences $\alpha=\{\alpha_ k\}\in T^p_{\infty}(Z)$ and $\beta=\{\beta_ k\}\in T^{\infty}_ q(Z)$, we have $\lambda=\alpha\cdot \beta \in T^p_ q(Z)$ with $$\|\lambda\|_{T^p_ q(Z)}\lesssim \|\alpha \|_{T^p_{\infty}(Z)}\cdot \|\beta \|_{T^{\infty}_ q(Z)}.$$
We will use the inequality
\begin{equation}\label{InC}
\sum_ k |\alpha_ k |^p \, |\beta_ k|^q \,(1-|a_ k|^2)^n \lesssim \|\alpha \|^p_{T^p_{\infty}(Z)}\cdot \|\beta \|^q_{T^{\infty}_ q(Z)},
\end{equation}
that can be proved in the same manner as some variants appearing in \cite{CMS,PR1} (see also \cite{PP}). We have
\begin{equation}\label{EqTB}
\| \alpha \cdot \beta \|^p_{T^p_ q(Z)}= \int_{\Sn} \left ( \sum_{a_ k \in \Gamma (\zeta)} |\alpha_ k |^q \,|\beta_ k|^q \right )^{p/q}\!\! d\sigma(\zeta).
\end{equation}
If $p/q>1$, consider a positive function $\varphi \in L^{p/(p-q)}(\Sn)$. Then
\[
\begin{split}
\int_{\Sn}  \Big (\sum_{a_ k \in \Gamma (\zeta)} |\alpha_ k |^q \,|\beta_ k|^q \Big )\,\varphi(\zeta) \,d\sigma(\zeta) &
\lesssim \int_{\Sn}  \Big (\sum_{k} |\alpha_ k |^q \,|\beta_ k|^q \,\frac{(1-|a_ k|^2)^{2n}}{|1-\langle \zeta,a_ k \rangle |^{2n}}\Big )\,\varphi(\zeta) \,d\sigma(\zeta)\\
&= \sum_ k |\alpha_ k |^q \,|\beta_ k|^q \,(1-|a_ k|^2)^{n} \, P\varphi (a_ k),
\end{split}
\]
where $P\varphi$ denotes the Poisson integral of $\varphi$. Recall that, when $\beta \in T^{\infty}_ q$, the measure $\mu_ {\beta}= \sum_ k |\beta_ k|^q(1-|a_ k|^2)^n \delta_{a_ k}$ is a Carleson measure with $\|\mu_ {\beta}\|_{CM_ 1} \asymp \|\beta \|^q_{T^{\infty}_ q(Z)}$. Hence
\[
\sum_ k |P\varphi (a_ k) |^{p/(p-q)} |\beta_ k|^q \,(1-|a_ k|^2)^n \lesssim   \|\beta \|^q_{T^{\infty}_ q(Z)} \cdot \|\varphi \|^{p/(p-q)}_{L^{p/(p-q)}(\Sn)}
\]
This together with the estimate \eqref{InC} and H\"{o}lder's inequality yield
\[
\begin{split}
\int_{\Sn}  \sum_{a_ k \in \Gamma (\zeta)} &|\alpha_ k |^q \,|\beta_ k|^q \,\varphi(\zeta) \,d\sigma(\zeta)
\\
&\le \left (\sum_ k |P\varphi (a_ k) |^{p/(p-q)} |\beta_ k|^q \,(1-|a_ k|^2)^n\right )^{(p-q)/p} \left ( \sum_ k |\alpha_ k |^p \, |\beta_ k|^q \,(1-|a_ k|^2)^n \right )^{q/p}
\\
& \lesssim \|\varphi \|_{L^{p/(p-q)}(\Sn)}\cdot \|\alpha \|_{T^p_{\infty}(Z)}^q \cdot \|\beta\|_{T^{\infty}_ q(Z)}^{q}.
\end{split}
\]
By duality, we obtain that $\|\alpha \cdot \beta \|_{T^p_ q(Z)}\lesssim \|\alpha \|_{T^p_{\infty}(Z)} \cdot \|\beta\|_{T^{\infty}_ q(Z)}$ when $p>q$. If $p=q$, this inequality is a direct consequence of the estimate \eqref{InC}.

Finally, consider the case that $p<q$. In this case, starting from \eqref{EqTB}, we use H\"{o}lder's inequality with exponent $q/p>1$ (that has conjugate exponent $q/(q-p)$), and then the estimate \eqref{InC} in order to get
\[
\begin{split}
\|\alpha \cdot \beta \|^p_{T^p_ q(Z)} & \le \int_{\Sn} \left (\sup_{a_ k\in \Gamma(\zeta)} |\alpha_ k| \right )^{\frac{p(q-p)}{q}}\left ( \sum_{a_ k \in \Gamma (\zeta)} |\alpha_ k |^p \,|\beta_ k|^q \right )^{p/q}\!\! d\sigma(\zeta)
\\
& \le \|\alpha \|_{T^p_{\infty}(Z)}^{\frac{p(q-p)}{q}} \left (\int_{\Sn}  \sum_{a_ k \in \Gamma (\zeta)} |\alpha_ k |^p \,|\beta_ k|^q \, d\sigma(\zeta)\right )^{p/q}
\\
& \asymp \|\alpha \|_{T^p_{\infty}(Z)}^{\frac{p(q-p)}{q}} \left ( \sum_{k} |\alpha_ k |^p \,|\beta_ k|^q \, (1-|a_ k|^2)^n \right )^{p/q}
\\
& \lesssim \|\alpha \|_{T^p_{\infty}(Z)}^{p} \cdot \|\beta\|_{T^{\infty}_ q(Z)}^p.
\end{split}
\]

\mbox{}
\\
Next, assume that $\lambda \in T^p_q(Z)$ and let $0<s<p$. We define the function
$$\alpha(z)=\left(\frac{1}{\sigma(I(z))}\int_{I(z)}\Big(\sum_{a_j \in \Gamma(\xi)}|\lambda_j|^q\Big)^{s/q}d\sigma(\xi)\right)^{1/s},$$
and set $\alpha_k=\alpha(a_k)$ for $a_k \in Z$. We want to show that $\{\alpha_k\} \in T^p_\infty(Z)$, which is equivalent to
$$\xi \mapsto \sup_{a_k \in \Gamma(\xi)}|\alpha_k|^s \in L^{p/s}(\Sn).$$
Writing
$$A_{\lambda}(\xi)=\sum_{a_j \in \Gamma(\xi)}|\lambda_j|^q,$$
we note that $|\alpha(z)|^s \lesssim P[A_{\lambda}^{s/q}](z)$, which follows easily by noting that when $\xi \in I(z)$, one has $|1-\langle z,\xi\rangle|\asymp 1-|z|^2$. Therefore, according to Theorem \ref{NPM},
$$\sup_{a_k \in \Gamma(\xi)}|\alpha_k|^s \lesssim \sup_{z \in \Gamma(\xi)}|\alpha(z)|^s \lesssim N[P[A_{\lambda}^{s/q}]](\xi)\lesssim M[A_{\lambda}^{s/q}](\xi).$$
Since $p/s>1$, an application of the Hardy-Littlewood theorem \eqref{hlm} gives us that $\{\alpha_k\} \in T^p_\infty(Z)$ and that $\|\{\alpha_k\}\|_{T^p_\infty(Z)}\lesssim \|\lambda\|_{T^p_q(Z)}$.\\

We now set $\beta_k=\lambda_k/\alpha_k$ and show that $\{\beta_k\}\in T^\infty_q(Z)$: that is
$$\mu_\beta=\sum_{k} |\beta_k|^q (1-|a_k|^2)^n \delta_{a_k}$$ is a Carleson measure. Let $\varepsilon>0$ be chosen so that $q/\varepsilon>1$, $s/\varepsilon>1$ and
$$s\geq \frac{q\varepsilon}{q-\varepsilon}.$$
If now $d\nu$ is a probability measure on some space $\Omega$, then for any non-negative measurable function $f$, H\"older's inequality yields
\begin{align*}
1=\left(\int_\Omega f^{-\varepsilon}f^{\varepsilon}d\nu\right)^{1/\varepsilon}&
\leq \left(\int_\Omega f^{-q}d\nu\right)^{1/q}\left(\int_\Omega f^\frac{q\varepsilon}{q-\varepsilon}d\nu\right)^{\frac{q-\varepsilon}{q\varepsilon}}\\
&\leq \left(\int_\Omega f^{-q}d\nu\right)^{1/q}\left(\int_\Omega f^s d\nu\right)^{1/s}.
\end{align*}
Using this estimate in case that $\Omega=I(z)$, $d\nu=\sigma(I(z))^{-1}d\sigma$ and $f=A_{\lambda}^{1/q}$,
gives us
$$|\alpha(z)|^{-q}\lesssim \frac{1}{\sigma(I(z))}\int_{I(z)}\frac{d\sigma(\xi)}{A_{\lambda}(\xi)}.$$

 Now, for $a \in \Bn$, we have
\begin{align*}
\int_{\Bn}\frac{(1-|a|^2)^n}{|1-\langle z,a\rangle|^{2n}}d\mu_\beta(z)&=\sum_{k} \frac{(1-|a|^2)^n}{|1-\langle a_k,a\rangle|^{2n}}\,|\lambda_k|^q \,|\alpha_k|^{-q}(1-|a_k|^2)^n\\
&\lesssim (1-|a|^2)^n\sum_k \frac{|\lambda_k|^q}{|1-\langle a_k,a\rangle|^{2n}}\int_{I(a_k)}\frac{d\sigma(\xi)}{A_{\lambda}(\xi)}\\
&\asymp (1-|a|^2)^n \int_{\Sn}\frac{1}{A_{\lambda}(\xi)}\sum_{a_k \in \Gamma(\xi)}\frac{|\lambda_k|^q}{|1-\langle a_k,a\rangle|^{2n}}\,d\sigma(\xi)\\
&\lesssim (1-|a|^2)^n \int_{\Sn}\frac{1}{|1-\langle \xi,a\rangle|^{2n}}\frac{1}{A_{\lambda}(\xi)}\Big (\sum_{a_k \in \Gamma(\xi)}|\lambda_k|^q\Big ) \,d\sigma(\xi)\\
&= (1-|a|^2)^n \int_{\Sn}\frac{d\sigma(\xi)}{|1-\langle \xi,a\rangle|^{2n}}<\infty.
\end{align*}
The last inequality follows from Lemma \ref{IctBn}. We immediately obtain that $\{\beta_k\} \in T^\infty_q(Z)$ with the estimate $\|\{\beta_k\}\|_{T^\infty_q(Z)}\lesssim 1$. Therefore, given $\lambda \in T^p_ q(Z)$, there are sequences $\alpha \in T^p_{\infty}(Z)$ and $\beta \in T^{\infty}_ q(Z)$ such that $\lambda =\alpha \cdot \beta$ with $\|\alpha \|_{T^p_{\infty}(Z)} \cdot \|\beta \|_{T^{\infty}_ q(Z)} \lesssim \|\lambda \|_{T^p_ q(Z)}$. The proof is complete.
\end{proof}

\begin{theorem}\label{th1}
Let $\alpha>-1$, $0<q<\infty$ and $p>\max\{2,q\}$. Then $J_b:A^p_\alpha\to H^q$ is bounded if and only if the function
$$U_b(\xi):=  \left(\int_{\Gamma(\xi)}|Rb(z)|^{\frac{2p}{p-2}}
(1-|z|^2)^{\frac{2-2\alpha}{p-2}+1-n}dV(z)\right)^{\frac{p-2}{2p}}$$
belongs to $L^{\frac{pq}{p-q}}(\Sn)$. Moreover,  $\|J_b\|_{A_\alpha^p\to H^q} \asymp \|U_b\|_{L^\frac{pq}{p-q}(\Sn)}$.
\end{theorem}

\begin{proof}
Since $p/(p-q)>1$, it will be convenient for us to prove the equivalent claim in terms of $U_b^q \in L^{\frac{p}{p-q}}(\Sn).$

Let us first take care of the sufficiency part. Using the area function description of the Hardy spaces and H\"older's inequalities, we have
\begin{align*}
\|J_b f\|_{H^q}^q \lesssim &\int_{\Sn}
\left(\int_{\Gamma(\xi)}|Rb(z)|^2 |f(z)|^2 (1-|z|^2)^{1-n}dV(z)\right)^{q/2}d\sigma(\xi)\\
\leq &\int_{\Sn}\left(\int_{\Gamma(\xi)}
|Rb(z)|^{\frac{2p}{p-2}}(1-|z|^2)^{\frac{p-2\alpha}{p-2}-n}
dV(z)\right)^{\frac{q(p-2)}{2p}}\\
&\cdot \left(\int_{\Gamma(\xi)}|f(z)|^p (1-|z|^2)^{\alpha-n}dV(z)\right)^{\frac{q}{p}}d\sigma(\xi)\\
\lesssim &\left(\int_{\Sn}U_ b(\xi)^{\frac{pq}{p-q}}\,
d\sigma(\xi)\right)^{\frac{p-q}{p}}\|f\|_{A^p_\alpha}^q.
\end{align*}
This proves the sufficiency part with
 $\|J_b\|_{A_\alpha^p\to H^q} \lesssim \|U_b\|_{L^\frac{pq}{p-q}(\Sn)}$.\\

Let us now consider the necessity. {Suppose that
$J_b:A_\alpha^p\to H^q$
is bounded.}
A straightforward approximation argument (see Lemma 7 in \cite{Per}) using the dilations $b_\rho(z)=b(\rho z)$ ($0<\rho<1$) shows that it suffices to establish the corresponding estimate in the case where the left-hand-side expression in Lemma \ref{disc_cont} is finite. {By Lemma \ref{disc_cont} it is sufficient to prove that}
\begin{equation}\label{disc}
\xi \mapsto \left(\sum_{a_k \in \Gamma(\xi)}|Rb(a_k)|^{\frac{2p}{p-2}}{
(1-|a_k|^2)^{\frac{2-2\alpha}{p-2}+2}}\right)^{\frac{q(p-2)}{2p}} \in L^{\frac{p}{p-q}}(\Sn),
\end{equation}
when $Z=\{a_k\}$ is an $r$-lattice and $r$ is small enough.

Now consider the test functions
$$F_t(z)=\sum_k (1-|a_k|^2)^{n/p}\lambda_k r_k(t)f_k(z),$$
where $\lambda=\{\lambda_k\} \in T^p_p(Z)$, $r_k:[0,1]\to \{-1,+1\}$ are the Rademacher functions, and
$$f_k(z)=\frac{(1-|a_k|^2)^{b-(n+1+\alpha)/p}}{(1-\langle z,a_k\rangle)^b}$$ are suitable kernel functions with $b>(n+1+\alpha)/p$. Notice that $\lambda \in T^p_p(Z)$ if and only if $((1-|a_k|^2)^{n/p}\lambda_k)\in \ell^p$. We could naturally work with $\ell^p$, but we will use a factorization trick, which is more transparent in the tent space notation.

By the area function description of the Hardy spaces,  the construction of $F_t$ and the atomic decomposition of the Bergman spaces, we have for every $t\in [0,1]$,
$$\int_{\Sn} \!\!\left(\int_{\Gamma(\xi)}\Big |Rb(z)\sum_k (1-|a_k|^2)^{n/p}\lambda_k r_k(t)f_k(z)\Big |^2 dV_{1-n}(z)\right)^{q/2} \!\!\! d\sigma(\xi)\lesssim \|J_b\|_{A^p_\alpha\to H^q}^q \|\lambda\|_{T^p_p(Z)}^q.$$
Integrating with respect to $t$ and using Fubini's theorem, we get
\begin{align*}
\int_{\Sn} \int_0^1 &\left(\int_{\Gamma(\xi)}\Big |Rb(z)\sum_k (1-|a_k|^2)^{n/p}\lambda_k r_k(t)f_k(z)\Big |^2 dV_{1-n}(z)\right)^{q/2} \!\!dt \,d\sigma(\xi)\\
&\lesssim \|J_b\|_{A^p_\alpha\to H^q}^q \cdot \|\lambda\|_{T^p_p(Z)}^q.
\end{align*}
Next, we use Kahane's inequality and Fubini's theorem to obtain
\begin{align*}
\int_{\Sn} &\left( \int_{\Gamma(\xi)}\int_0^1 \Big |Rb(z)\sum_k (1-|a_k|^2)^{n/p} \lambda_k r_k(t)f_k(z)\Big |^2 dt \,dV_{1-n}(z) \right)^{q/2} d\sigma(\xi)\\
&\lesssim \|J_b\|_{A^p_\alpha\to H^q}^q \cdot \|\lambda\|_{T^p_p(Z)}^q.
\end{align*}
Further, using Khinchine's inequality gives that
\[
\int_{\Sn} \!\!\left( \sum_k |\lambda_k|^2 \!\int_{\Gamma(\xi)} \!\!|Rb(z)|^2 (1-|a_k|^2)^{2n/p}|f_k(z)|^2 dV_{1-n}(z) \right)^{q/2} \!\! \!\! d\sigma(\xi)
\lesssim \|J_b\|_{A^p_\alpha\to H^q}^q \cdot \|\lambda\|_{T^p_p(Z)}^q.
\]

\noindent For $\theta>n\max\{2/q,1\}$, using Lemma \ref{Gamma} we obtain
\begin{align*}
\int_{\Sn} &\left( \sum_k |\lambda_k|^2 \int_{\Bn} \left(\frac{(1-|z|^2)}{|1-\langle z,\xi\rangle|}\right)^\theta |Rb(z)|^2 (1-|a_k|^2)^{2n/p}|f_k(z)|^2 dV_{1-n}(z) \right)^{q/2} \!\! d\sigma(\xi)\\
\\
\nonumber &\lesssim \|J_b\|_{A^p_\alpha\to H^q}^q \cdot \|\lambda\|_{T^p_p(Z)}^q.
\end{align*}
Replacing the integral over $\Bn$ by an integral over $D(a_k,r)$, and using subharmonicity and standard estimates, we arrive at

\begin{align*}
\int_{\Sn} &\!\!\left( \sum_k |\lambda_k|^2 \left(\frac{(1-|a_k|^2)}{|1-\langle a_k,\xi\rangle|}\right)^\theta |Rb(a_k)|^2 (1-|a_k|^2)^{2-2(1+\alpha)/p} \right)^{q/2} \!\!\!d\sigma(\xi)\\
\\
&\lesssim \|J_b\|_{A^p_\alpha \to H^q}^q\cdot \|\lambda\|_{T^p_p(Z)}^q.
\end{align*}

\noindent Finally, for every $\xi \in \Sn$, summing over only the points $a_k \in \Gamma(\xi)$ (and noting that $a_k \in \Gamma(\xi)$ implies that $(1-|a_k|^2)\asymp |1-\langle a_k,\xi\rangle |$) gives that
\begin{equation}\label{upper}
\int_{\Sn} \!\!\left( \sum_{a_k \in \Gamma(\xi)} \!|\lambda_k|^2 |Rb(a_k)|^2 (1-|a_k|^2)^{2-2(1+\alpha)/p} \right)^{q/2} \!\!\!d\sigma(\xi)\lesssim \|J_b\|_{{A^p_\alpha}\to H^q}^q\cdot \|\lambda\|_{T^p_p(Z)}^q.
\end{equation}

{Recall that  by \eqref{disc}, we want} to prove that
$$\xi \mapsto \left(\sum_{a_k \in \Gamma(\xi)} |Rb(a_k)|^{\frac{2p}{p-2}}(1-|a_k|^2)^{\frac{2-2\alpha}{p-2}+2}\right)^{\frac{q(p-2)}{2p}}$$
belongs to $L^{p/(p-q)}(\Sn)$. Write $\nu=\{\nu_k\}$, where
$$\nu_k=|Rb(a_k)|^q (1-|a_k|^2)^{\frac{q}{p}(p-1-\alpha)}.$$
{This means that we want to prove}
$$\nu \in T^{\frac{p}{p-q}}_\eta(Z), \quad \eta=\frac{2p}{q(p-2)}.$$
For every $s>1$, this is equivalent to the statement $\nu^{1/s}\in T^{ps/(p-q)}_{\eta s}(Z)$.
For $s$ large enough, we have
$$T^{\frac{ps}{p-q}}_{\eta s}(Z)=\left(T^{\left(\frac{ps}{p-q}\right)'}_{(\eta s)'}(Z)\right)^*=\left(T^{s'}_{\frac{2s}{2s-q}}(Z) \cdot T^{\frac{ps}{q}}_{\frac{ps}{q}}(Z)\right)^*,$$
where the last identity is the factorization of the corresponding tent spaces of sequences given in Proposition \ref{factor}.

Let us take $$\mu=\{\mu_k\}\in T^{\left(\frac{ps}{p-q}\right)'}_{(\eta s)'}(Z),$$
and factor it as suggested
$$\mu_k= \tau_k \,\lambda_k^{q/s}, \quad \tau=\{\tau_k\} \in T^{s'}_{\frac{2s}{2s-q}}(Z), \quad \lambda  \in T^p_p(Z).$$
Then (note that we may, without loss of generality, assume that all sequences are positive)
\begin{align*}
\sum_k \mu_k \, \nu_k^{1/s} &\,(1-|a_k|^2)^n =\sum_k \tau_k \,\lambda_k^{q/s}\,\nu_k^{1/s}\,(1-|a_k|^2)^n\\
&\asymp \int_{\Sn} \Big(\sum_{a_k \in \Gamma(\xi)} \tau_k \, \lambda_k^{q/s}\,\nu_k^{1/s}\Big)\,d\sigma(\xi)\\
&\leq \int_{\Sn} \Big(\sum_{a_k \in \Gamma(\xi)}\tau_k^{(2s/q)'}\Big)^{1-q/2s} \Big(\sum_{a_k \in \Gamma(\xi)} \lambda_k^2 \,\nu_k^{2/q}\Big)^{q/2s}\,d\sigma(\xi)\\
&\leq \|\tau\|_{T^{s'}_{\frac{2s}{2s-q}}(Z)}\left(\int_{\Sn} \big(\sum_{a_k \in \Gamma(\xi)} \lambda_k^2 \,\nu_k^{2/q}\Big)^{q/2}d\sigma(\xi)\right)^{1/s}.
\end{align*}
Here the two last inequalities are just applications of H\"older {inequalities with $2s/q$ and $s$, respectively, which are legal by making $s$ large enough.}

Analyzing the second factor by using our earlier estimate \eqref{upper}, we establish
$$\sum_k \mu_k \,\nu_k^{1/s}(1-|a_k|^2)^n \lesssim \|J_b\|_{A^p_\alpha\to H^q}^{q/s}\cdot \|\tau\|_{T^{s'}_{\frac{2s}{2s-q}}(Z)}\cdot  {\|\lambda\|_{T^{{p}}_{{p}}(Z)}^{q/s}}.$$
Taking infimum over all possible factorizations yields
$$\sum_k \mu_k \nu_k^{1/s}(1-|a_k|^2)^n \lesssim \|J_b\|_{A^p_\alpha\to H^q}^{q/s} \cdot \|\mu\|_{T^{\left(\frac{ps}{p-q}\right)'}_{(\eta s)'}(Z)}.$$
We obtain \eqref{disc} by the duality of tent spaces of sequences and $ \|U_b\|_{L^\frac{pq}{p-q}(\Sn)}\lesssim\|J_b\|_{A_\alpha^p\to H^q}$. The proof is therefore completed.
\end{proof}

The following result will complete our main theorem.

\begin{theorem}\label{th11}
Let $0<q<p\leq 2$ and $\alpha>-1$. Then $J_b:A^p_\alpha\to H^q$ is bounded if and only if the function
\[
V_b(\xi):= \sup_{z \in \Gamma(\xi)}|Rb(z)| (1-|z|^2)^{\frac{p-1-\alpha}{p}}
\]
belongs to $L^{\frac{pq}{p-q}}(\Sn).$ Moreover,  $\|J_b\|_{A_\alpha^p\to H^q}\asymp \|V_b\|_{L^{\frac{pq}{p-q}}(\Sn)}$.
\end{theorem}

\begin{proof}
Let us first prove the sufficiency. By Theorem \ref{area} and H\"older's inequality with exponent $p/q>1$, we have
\begin{align*}
\|J_b f\|_{H^q}^q &\lesssim \int_{\Sn}\left(\int_{\Gamma(\xi)}|Rb(z)|^2 |f(z)|^2 (1-|z|^2)^{1-n}dV(z)\right)^{q/2}d\sigma(\xi)\\
&\leq \int_{\Sn}V_b(\xi)^q \left(\int_{\Gamma(\xi)}|f(z)|^2 (1-|z|^2)^{1-n-\frac{2}{p}(p-1-\alpha)}dV(z)\right)^{q/2}d\sigma(\xi)\\
&\leq \|V_ b\|_{L^{\frac{pq}{p-q}}(\Sn)}^q \left (\int_{\Sn}\left(\int_{\Gamma(\xi)}|f(z)|^2 (1-|z|^2)^{1-n-\frac{2}{p}(p-1-\alpha)}dV(z)\right)^{p/2} \!\! d\sigma(\xi)\right )^{q/p}.
\end{align*}
We estimate the second factor of the line above:
\begin{align*}
\int_{\Gamma(\xi)}|f(z)|^2 & (1-|z|^2)^{1-n-\frac{2}{p}(p-1-\alpha)}dV(z)\\
&\lesssim \sum_{a_k \in \widetilde{\Gamma}(\xi)}(1-|a_k|^2)^{1-n-\frac{2}{p}(p-1-\alpha)}\int_{D(a_k,r)}|f(z)|^2 dV(z)\\
&\lesssim \sum_{a_k \in \widetilde{\Gamma}(\xi)}(1-|a_k|^2)^{2-\frac{2}{p}(p-1-\alpha)-\frac{2}{p}(n+1)}\left(\int_{D(a_k,2r)}|f(z)|^p dV(z)\right)^{2/p}.
\end{align*}
Since $p/2\le 1$, we get
\begin{align*}
\left(\int_{\Gamma(\xi)}\!\!|f(z)|^2 (1-|z|^2)^{1-n-\frac{2}{p}(p-1-\alpha)}dV(z)\right)^{p/2}&
\lesssim \!\!\sum_{a_k \in \widetilde{\Gamma}(\xi)}\!(1-|a_k|^2)^{\alpha-n}\int_{D(a_k,2r)}\!\!|f(z)|^p dV(z)\\
&\lesssim \int_{\widetilde{\widetilde{\Gamma}}(\xi)}|f(z)|^p (1-|z|^2)^{\alpha-n}dV(z).
\end{align*}

Combining the estimates above gives that
\begin{align*}
\|J_b f\|_{H^q}^q &
\lesssim \|V_b\|_{L^{\frac{pq}{p-q}}(\Sn)}^q \left(\int_{\Sn} \left(\int_{\widetilde{\widetilde{\Gamma}}(\xi)}|f(z)|^p (1-|z|^2)^{\alpha-n}dV(z)\right)d\sigma(\xi)\right)^{q/p}\\
&\lesssim \|V_b\|_{L^{\frac{pq}{p-q}}(\Sn)}^q \left(\int_{\Bn} |f(z)|^p (1-|z|^2)^{\alpha}dV(z)\right)^{q/p}
\end{align*}
This yields the sufficiency of the condition
\begin{equation}\label{suffpbig}
\xi \mapsto \sup_{z \in \Gamma(\xi)}|Rb(z)| (1-|z|^2)^{\frac{p-1-\alpha}{p}} \in L^{\frac{pq}{p-q}}(\Sn).
\end{equation}
And moreover, $\|J_b\|_{A_\alpha^p\to H^q}\lesssim \|V_b\|_{L^{\frac{pq}{p-q}}(\Sn)}$.\\

Let us now turn to the necessity. By reasoning similar to the corresponding point in the proof of Theorem \ref{th1}, we may use an approximation argument. Let us assume that our $r$-lattice $\{a_k\}$ satisfies Lemma \ref{discsup}. We can follow that argument of necessity in Theorem \ref{th1}, using Kahane's inequality and Khinchine's inequality together with similar techniques, leading to estimate \eqref{upper}. By writing $\nu=\{\nu_k\}$, where
$$\nu_k=|Rb(a_k)|^q (1-|a_k|^2)^{\frac{q}{p}(p-1-\alpha)},$$
{it is sufficient to show that}
$$\nu \in T^{p/(p-q)}_\infty(Z).$$
For every $s>1$, this is equivalent to
$${\nu^{1/s}:=(\nu_k^{1/s})}\in T^{ps/(p-q)}_\infty(Z)=\left(T^{s'}_{\frac{2s}{2s-q}}(Z) \cdot T^{\frac{ps}{q}}_{\frac{ps}{q}}(Z)\right)^*,$$
by the factorization result in Proposition \ref{factor}. The reader should notice here that if $p\leq 2$, then
$$\frac{q}{ps}+\frac{2s-q}{2s}=1/\rho$$
for some $\rho\leq 1$.
so the factorization and the duality can be carried over.

Now, the argument can be completed by carrying out the same reasoning as in the corresponding place in the proof of Theorem \ref{th1}. We only need to choose $s$ large enough such that all applications of H\"older's inequality are valid. The necessity will follow by virtue of Lemma \ref{discsup}, and then the proof is complete.
\end{proof}


By combining  the theorems in Sections \ref{pleq} and \ref{pbig}, we  obtain the whole proof of our main result.\medskip

\section{Concluding remarks}

\subsection{Trivial case}

Looking at Theorem \ref{main}, it is important to know when the conditions are void, that is, when the only $b \in H$ generating bounded $J_b$ are the constant functions. In this case, we always have $J_b=0$. The following corollary follows easily from the main theorem.

\begin{coro}
Let $\alpha>-1$, $0<p,q<\infty$ and $b \in H$. Then the following hold:
\begin{itemize}
\item[(1) ]  {Let} $0<p<q<\infty$. If $\alpha-p>n(p/q-1)-1$,
then $J_b:A^p_\alpha\to H^q$ is bounded if and only if $J_b=0$.

\item[(2) ] {Let} $0<q\le p<\infty$. If $\alpha-p> -1$, then $J_b:A^p_\alpha\to H^q$ is {bounded} if and only if $J_b=0$.
\end{itemize}
\end{coro}

Since $A^p_\alpha$ {decreases with respect to $p$, and increases with respect to  $\alpha$,} it should not come as a surprise that the deciding factor is the quantity $\alpha-p$.

\subsection{Comparison with $J_b$ acting between Hardy spaces}

Since $H^p$ often appears as the limit of $A^p_\alpha$ when $\alpha\to -1$, it makes sense to compare our results with those of \cite{P1} where the boundedness of $J_b:H^p\to H^q$ is described. We list some remarks, which should be taken only in the heuristical sense.

\begin{itemize}
\item[(1) ] If $0<p\leq \min\{2,q\}$ or $2<p<q<\infty$, then as $\alpha\to -1$, the characterizing condition becomes
$$\sup_{z \in \Bn} |Rb(z)|(1-|z|^2)^{\frac{n}{q}-\frac{n}{p}+1}<\infty.$$
This agrees exactly with the $H^p$ case, when $0<p<q<\infty$. When $p=q\leq 2$, this becomes the Bloch condition, which is slightly weaker than $BMOA$ condition for the $H^p$ case.
\item[(2) ] If $2<p=q<\infty$, then as $\alpha\to -1$, the characterizing condition becomes that $d\mu_b(z)=|Rb(z)|^{\frac{2p}{p-2}}(1-|z|^2)^{\frac{p+2}{p-2}} dV(z)$ is a Carleson measure. In the $H^p$ case one has the $BMOA$ condition, which is equivalent to
$|Rb(z)|^2 (1-|z|^2)dV(z)$
being a Carleson measure. Note that $2p/(p-2)>2$, but the conditions become the same as $p\to \infty$ (i.e., $2p/(p-2)\to 2$).
\item[(3) ] If $p>\max\{2,q\}$, then as $\alpha\to -1$, the characterizing condition becomes that
$$\xi \mapsto \left(\int_{\Gamma(\xi)}|Rb(z)|^{\frac{2p}{p-2}}(1-|z|^2)^{\frac{4}{p-2}+1-n}dV(z)\right)^{\frac{(p-2)}{2p}}$$
belongs to $L^{\frac{pq}{p-q}}(\Sn)$. Similarly to the previous case, this agrees with the $H^p$ case in the limit when $p\to \infty$. Indeed, by Theorem \ref{area}, this agrees with $b \in H^{\frac{pq}{p-q}}$.
\item[(4) ] If $0<q<p\leq 2$, then as $\alpha\to -1$, the characterizing condition becomes that the function
$\xi \mapsto \sup_{z \in \Gamma(\xi)}|Rb(z)| (1-|z|^2)$
belongs to $L^{\frac{pq}{p-q}}(\Sn)$. In the terminology and notation of \cite{Per}, this means that $b \in \mathcal{BT}^{\frac{pq}{p-q}}$. It is not difficult to see that for any $\beta>-1$, one has
$H^{\frac{pq}{p-q}}\subset \mathcal{BT}^{\frac{pq}{p-q}} \subset A^{\frac{pq}{p-q}}_\beta,$
and no inclusion can be reversed.
\end{itemize}

\subsection{Integration operators from Hardy to Bergman spaces}
It is also interesting to study, for $b\in H$, the boundedness of $J_ b:H^p\rightarrow A^q_{\alpha}$ for $\alpha>-1$ and $0<p,q<\infty$. We can always assume that $b(0)=0$, and by the characterization of Bergman spaces in terms of radial derivatives, we have
$
\|J_ b f \|^q_{A^q_{\alpha}} \asymp \|R(J_ bf )\|^q_{A^q_{\alpha+q}}.
$
Using the basic identity $R(J_ bf)=fRb$ we see that the boundedness is equivalent to
\[
\int_{\Bn} |f(z)|^q \,|Rb(z)|^q \,(1-|z|^2)^{q+\alpha} dV(z) \lesssim \|f\|^q_{H^p}.
\]
In other words, we have the embedding from $H^p$ into $L^q(\mu)$ with $d\mu(z)=|Rb(z)|^q \,(1-|z|^2)^{q+\alpha} dV(z)$. Thus, by the Carleson-Duren and Luecking's theorems we get the following result.
\begin{theorem}
Let $b\in H$ and $\alpha>-1$. Then
\begin{itemize}
\item[(i)] For $0<p<\infty$, the operator $J_ b:H^p\rightarrow A^p_{\alpha}$ is bounded if and only if $d\mu(z)=|Rb(z)|^p \,(1-|z|^2)^{p+\alpha} dV(z)$ is a Carleson measure.

\item[(ii)] For $0<p<q<\infty$, the operator $J_ b:H^p\rightarrow A^q_{\alpha}$ is bounded if and only if $b$ belongs to the Bloch type space $\mathcal{B}^{\gamma}$ with $\gamma=1+\frac{(n+1+\alpha)}{q} - \frac{n}{p}$

\item[(iii)] For $0<q<p<\infty$, the operator $J_ b:H^p\rightarrow A^q_{\alpha}$ is bounded if and only if the function
\[
H_ b(\zeta):= \int_{\Gamma(\zeta)} |Rb(z)|^q \,(1-|z|^2)^{q+\alpha-n} dV(z)
\]
belongs to $L^{p/(p-q)}(\Sn)$.
\end{itemize}

\end{theorem}
In the one-dimensional case, part of the previous results were obtained by J. R\"{a}tty\"{a} in \cite{Rat}.

\end{document}